\documentclass[11pt]{amsart}
\usepackage{amsmath, amssymb, amsfonts, amsthm, graphicx}

\newcommand{\conv}{\mathrm{conv}\ }
\newcommand{\epsi}{\varepsilon}
\newcommand{\norm}[1]{\|#1\|}
\newcommand{\N}{\mathbf{N}}
\newcommand{\abs}[1]{\left|#1\right|}

\newtheorem{theorem}{Theorem}
\newtheorem{mylemma}[theorem]{Lemma}
\newtheorem*{thm}{Theorem}

\begin{document}

\title{Balancing unit vectors}


\author{Konrad J. Swanepoel}
\address{Department of Mathematics and Applied Mathematics, 
University of Pretoria, Pretoria 0002, South Africa}

\email{konrad@math.up.ac.za}

\begin{abstract}
 \mbox{}

\textbf{Theorem A.} Let $x_1,\dots,x_{2k+1}$ be unit vectors in a 
normed plane.
Then there exist signs $\epsi_1,\dots,\epsi_{2k+1}\in\{\pm 1\}$ such that
$\norm{\sum_{i=1}^{2k+1}\epsi_i x_i}\leq 1$.

We use the method of proof of the above theorem to show the following point 
facility location result, generalizing Proposition 6.4 of Y. S. Kupitz and 
H. Martini (1997).

\textbf{Theorem B.} Let $p_0,p_1,\dots,p_n$ be distinct points in a normed 
plane such that for any $1\leq i<j\leq n$ the closed angle $\angle p_ip_0p_j$ 
contains a ray opposite some $\overrightarrow{p_0p_k}, 1\leq k\leq n$.
Then $p_0$ is a Fermat-Toricelli point of $\{p_0,p_1,\dots,p_n\}$, i.e.\
$x=p_0$ minimizes $\sum_{i=0}^n\norm{x-p_i}$.

We also prove the following dynamic version of Theorem A.

\textbf{Theorem C.} Let $x_1,x_2,\dots$ be a sequence of unit vectors in a 
normed plane.
Then there exist signs $\epsi_1,\epsi_2,\dots\in\{\pm 1\}$ such that
$\norm{\sum_{i=1}^{2k}\epsi_i x_i}\leq 2$ for all $k\in\N$.

Finally we discuss a variation of a two-player balancing game of J. Spencer 
(1977) related to Theorem C.
\end{abstract}

\maketitle

\section{Introduction}
In this note we consider balancing results for unit vectors related to work 
of B\'ar\'any and Grinberg \cite{BG}, Spencer \cite{Sp1} and Peng and Yan 
\cite{PY}.
We apply these results to generalize a point facility location result from 
the Euclidean plane \cite{KM} to general normed planes.
Finally we consider a dynamical balancing problem for unit vectors in the 
form of a two-player perfect information game.
Our results will mainly be in a normed plane $X$ with norm $\norm{\cdot}$ 
(except in Theorem~\ref{gametheorem}, where higher-dimensional normed spaces 
are also considered).

\subsection{Balancing Unit Vectors}
B\'ar\'any and Grinberg \cite{BG} proved the following:

\begin{theorem}[\cite{BG}]\label{bgone}
Let $x_1,x_2,\dots,x_n$ be a sequence of vectors of norm $\leq 1$ in a 
$d$-dimensional normed space.
Then there exist signs $\epsi_1,\epsi_2,\dots,\epsi_n\in\{\pm 1\}$ such that 
\[\norm{\sum_{i=1}^n\epsi_ix_i}\leq d.\]
\end{theorem}

We sharpen this theorem for an odd number of \emph{unit} vectors in a normed 
plane as follows.

\begin{theorem}\label{oddunitvectors}
Let $x_1,\dots,x_{2k+1}$ be unit vectors in a normed plane.
Then there exist signs $\epsi_1,\dots,\epsi_{2k+1}\in\{\pm 1\}$ such that
\[\norm{\sum_{i=1}^{2k+1}\epsi_i x_i}\leq 1.\]
\end{theorem}
This result is best possible in any norm, as is seen by letting 
$x_1=x_2=\dots=x_{2k+1}$ be any unit vector.
The proof of this theorem is in Section~\ref{zonosection}.
The method of proof can also be used to generalize a result on 
Fermat-Toricelli points from the Euclidean plane to an arbitrary normed plane 
(Section~\ref{ftsection}).

B\'ar\'any and Grinberg also proved the following dynamic balancing theorem.

\begin{theorem}[\cite{BG}]\label{bgtwo}
Let $x_1,x_2,\dots$ be a sequence of vectors of norm $\leq 1$ in a normed 
space.
Then there exist signs $\epsi_1,\epsi_2,\dots\in\{\pm 1\}$ such that for all 
$k\in\N$,
\[\norm{\sum_{i=1}^k\epsi_ix_i}\leq 2d.\]
\end{theorem}

Again, for unit vectors in a normed plane we sharpen this result as follows.

\begin{theorem}\label{thmtwo}
Let $x_1,x_2,\dots$ be a sequence of unit vectors in a normed plane.
Then there exist signs $\epsi_1,\epsi_2,\dots\in\{\pm 1\}$ such that for all 
$k\in\N$,
\[\norm{\sum_{i=1}^{2k}\epsi_i x_i}\leq 2.\]
In the Euclidean plane the upper bound $2$ can be replaced by $\sqrt{2}$.
\end{theorem}
This result is best possible in the rectilinear plane with unit ball a 
parallelogram --- let $x_{2i-1}=e_1$ and $x_{2i}=e_2$ for all $i\in\N$, where 
$e_1$ and $e_2$ are any adjacent vertices of the unit ball.
See Section~\ref{onlinesection} for a proof of this theorem.
\subsection{Balancing Games}
Theorem~\ref{thmtwo} can be used to analyze the following variation of a 
two-player balancing game of Spencer.
Fix $k\in\N$ and a normed space $X$.
Let the starting position of the game be $p_0=o\in X$.
In round $i$, Player I chooses $k$ unit vectors $x_1,\dots,x_k$ in $X$, and 
then Player II chooses signs $\epsi_1,\dots,\epsi_k\in\{\pm 1\}$.
Then the position is adjusted to $p_i:= p_{i-1}+\sum_{j=1}^k\epsi_jx_j$.

\begin{theorem}\label{gametheorem}
In the above game, Player II can keep the sequence $(p_i)_{i\in\N}$ bounded 
iff $X$ is at most two-dimensional and $k$ is even.
In fact, Player II can force $\norm{p_i}\leq 2$ for all $i\in\N$.
\end{theorem}
The proof is in Section~\ref{onlinesection}.
In \cite{PY} a vector balancing game with a buffer is considered.
Theorem~\ref{gametheorem} readily implies Theorem 4 of \cite{PY} in the 
special case of unit vectors in a normed plane.

\subsection{Fermat-Toricelli points}\label{ftsection}
A point $p$ in a normed space $X$ is a \emph{Fermat-Toricelli point} of 
$x_1,x_2,\dots,x_n\in X$ if $x=p$ minimizes $x\mapsto\sum_{i=1}^n\norm{x_i-x}$.
See \cite{KM} for a survey on the problem of finding such points.
It is well-known that in the Euclidean plane, if $x_1$ is in the convex hull 
of non-collinear $\{x_2,x_3,x_4\}$, then $x_1$ is the (unique) 
Fermat-Toricelli point of $x_1,x_2,x_3,x_4$.
Cieslik \cite{Cieslik2} generalized this result to an arbitrary normed plane 
(where the Fermat-Toricelli point is not necessarily unique).
There is also a generalization by Kupitz and Martini 
\cite[Proposition 6.4]{KM} in another direction.

\begin{thm}
Let $p_0,p_1,\dots,p_{2m+1}$ be distinct points in the Euclidean plane such 
that for any distinct $i$ and $j$ the open angle $\angle p_ip_0p_j$ contains 
a ray opposite some $\overrightarrow{p_0p_k}, 1\leq k\leq 2m+1$.
Then $p_0$ is the unique Fermat-Toricelli point of $\{p_0,p_1,\dots,p_n\}$.
\end{thm}

We generalize this result as follows to an arbitrary normed plane.

\begin{theorem}\label{ftresult}
Let $p_0,p_1,\dots,p_n$ be distinct points in a normed plane such that for 
any distinct $i$ and $j$ the closed angle $\angle p_ip_0p_j$ contains a ray 
opposite some $\overrightarrow{p_0p_k}, 1\leq k\leq n$.
Then $p_0$ is a Fermat-Toricelli point of $\{p_0,p_1,\dots,p_n\}$.
\end{theorem}
The proof is in Section~\ref{zonosection}.
Our seemingly weaker hypotheses easily imply that $n$ must be odd.
The proof in \cite{KM} of the Euclidean case uses rotations.
Our proof for any norm shows that it is really an affine result.
The correct affine tool turns out to be the fact that two-dimensional 
centrally symmetric polytopes are zonotopes.

\section{Zonogons}\label{zonosection}
A \emph{zonotope} $P$ in a $d$-dimensional vector space $X$ is a Minkowski 
sum of line segments
\[ P= [x_1,y_1] + [x_2,y_2] + \dots + [x_n,y_n]\]
where $x_1,\dots,x_n,y_1,\dots,y_n\in X$.
It is well-known that any centrally symmetric two-dimensional polytope (or 
polygon) is always a zonotope (or \emph{zonogon}) 
\cite[Example 7.14]{ZiegPolyt}.
In particular, if $x_1,\dots,x_n$ are consecutive edges of a $2n$-gon $P$ 
symmetric around $0$, then
\begin{equation}\label{zono}
P = \sum_{i=1}^n [(x_{i+1}-x_i)/2, (x_i-x_{i+1})/2]
\end{equation}
where we take $x_{n+1}=-x_1$.

\begin{mylemma}\label{zonolemma}
Let $n\in\N$ be odd and let $P$ be a polygon with vertices $\pm x_1$, $\dots,$ $\pm x_n$ 
with $x_1,\dots,x_n$ in this order on the boundary of $P$.
Then
\[ \sum_{i=1}^n (-1)^i x_i 
    = \frac{1}{2}\sum_{i=1}^n (-1)^{i+1}(x_{i+1}-x_i)\in P.\]
\end{mylemma}
\begin{proof}
The equation is simple to verify.
That the right-hand side is in $P$ follows from (\ref{zono}).
\end{proof}

Note that Lemma~\ref{zonolemma} does not hold for even $n$.
We can now easily prove Theorem~\ref{oddunitvectors}.

\begin{proof}[Proof of Theorem~\ref{oddunitvectors}]
Fix a line through the origin not containing any $x_i$.
Fix one of the open half planes $H$ bounded by this line.
Then for each $i$, $\delta_ix_i\in H$ for some $\delta_i\in\{\pm 1\}$.
We may renumber $x_1,\dots,x_n$ such that $\delta_1x_1,\dots,\delta_n x_n$ 
occur in this order on $P=\conv\{\pm x_i\}$.
Now take $\epsi_i=(-1)^i\delta_i$ and apply Lemma~\ref{zonolemma}, noting 
that $P$ is contained in the unit ball.
\end{proof}

Recall that the dual of a finite dimensional normed space $X$ is the normed 
space of all linear functionals on $X$ with norm 
$\norm{\phi}=\max\{\phi(u): \norm{u}=1\}$.
A \emph{norming functional} $\phi$ of a non-zero $x\in X$ is a linear 
functional satisfying $\norm{\phi}=1$ and $\phi(x)=\norm{x}$.
Recall that by the separation theorem any non-zero $x\in X$ has a norming 
functional (see e.g.\ \cite{Thompson}).

The following lemma is well-known and easily proved.
See \cite{KM} for the Euclidean case and \cite{DM} for the general case.
We only need the second case of the lemma, but we also state the first case 
for the sake of completeness.
\begin{mylemma}\label{ftchar}
Let $p_0,p_1,\dots,p_n$ be distinct points in a finite-dimensional normed space 
$X$.
\begin{enumerate}
\item Then $p_0$ is a Fermat-Toricelli point of $p_1,\dots,p_n$ iff $p_i-p_0$ 
has a norming functional $\phi_i$ $(1\leq i\leq n)$ such that 
$\sum_{i=1}^n\phi_i=o$,
\item and $p_0$ is a Fermat-Toricelli point of $p_0,p_1,\dots,p_n$ iff 
$p_i-p_0$ has a norming functional $\phi_i$ $(1\leq i\leq n)$ such that 
$\norm{\sum_{i=1}^n\phi_i}\leq 1$.
\end{enumerate}
\end{mylemma}

\begin{proof}[Proof of Theorem~\ref{ftresult}]
By Lemma~\ref{ftchar} it is sufficient to find norming functionals $\phi_i$ 
of $p_i-p_0$ such that $\norm{\sum_{i=1}^n\phi_i}\leq 1$.
We order $p_1,\dots,p_n$ such that $\overrightarrow{p_0p_1},\dots,\overrightarrow{p_0p_n}$ are 
ordered counter-clockwise.
If $p_0\in[p_i,p_j]$ for some $1\leq i<j\leq n$, we may choose 
$\phi_i=-\phi_j$.
We may therefore assume that $p_0\notin[p_i,p_j]$ for all distinct $i,j$.
Thus for any $i$, the \emph{open} angle $\angle p_ip_0p_{i+1}$ contains a ray 
opposite some $\overrightarrow{p_0p_k}$.
We now show that necessarily $n$ is odd and $k\equiv i+(n+1)/2 \pmod{n}$.
Since each open angle contains at least one $-p_k$, each open angle contains 
exactly one such $-p_k$, say $-p_{k(i)}$.
The line through $p_0$ and $p_{k(i)}$ cuts $\{p_1,\dots,p_n\}$ in two open 
half planes:
One half plane contains as many open angles as points $p_i$.
Thus $n$ is odd, and $k(i)\equiv i+(n+1)/2\pmod{n}$.

It is now possible to choose norming functionals $\phi_i$ of each $p_i-p_0$ 
such that $\phi_1,-\phi_{m+1},\phi_2,-\phi_{m+2},\dots$ are consecutive 
vectors on the unit circle in the dual normed plane.
It is therefore sufficient to prove that in any normed plane, if we choose 
unit vectors $x_1,\dots,x_n$ such that $x_1,\dots,x_n,-x_1,\dots,-x_n$ are in 
this order on the unit circle, then $\norm{\sum_{k=1}^n(-1)^kx_k}\leq 1$.
This follows at once from Lemma~\ref{zonolemma}.
\end{proof}

\section{Online Balancing}\label{onlinesection}
\begin{proof}[Proof of Theorem~\ref{gametheorem}]
$\Rightarrow$
We assume that some inner product structure has been fixed on $X$.

If $k$ is odd then in round $i$ Player I chooses the $k$ unit vectors all to 
be the same unit vector, orthogonal to $p_{i-1}$.
Then, independent of the choice of signs by Player II, the Euclidean norm of 
$p_i$ grows $> c\sqrt{i}$.

If $k$ is even and $X$ is at least three-dimensional, Player I finds unit 
vectors $e_1$ and $e_2$ such that $e_1,e_2,p_{i-1}$ are mutually orthogonal,
 then in round $i$ takes $e_1$ for the first $k-1$ unit vectors, and $e_2$ 
for the last unit vector.
Again the Euclidean norm of $p_i$ will grow  $>c\sqrt{i}$.

$\Leftarrow$ follows immediately from Lemmas~\ref{onlinelemma} and 
\ref{euclonlinelemma} below.
\end{proof}

\begin{proof}[Proof of Theorem~\ref{thmtwo}]
follows immediately from the following two lemmas.
\end{proof}

\begin{mylemma}\label{onlinelemma}
Let $w,a,b$ be vectors in a normed plane such that $\norm{w}\leq 2$, 
$\norm{a}=\norm{b}=1$.
Then there exist signs $\delta,\epsi\in\{\pm 1\}$ such that 
$\norm{w+\delta a+\epsi b}\leq 2$.
\end{mylemma}
\begin{proof}
If $a=\pm b$, then the lemma is trivial.
So assume that $a$ and $b$ are linearly independent.
Let $w=\lambda a+\mu b$.
Without loss of generality we assume that $\lambda,\mu\geq 0$, and show that 
$\norm{w-a-b}\leq 2$.

If $\lambda=0$, then $0\leq\mu\leq 2$ and 
$\norm{(\lambda-1)a+(\mu-1)b}\leq \norm{a}+\norm{(\mu-1)b}\leq 2$.
So we may assume that $\lambda>0$, and similarly, $\mu>0$.
Then we can write $a=-(\mu/\lambda)b+(1/\lambda)w$.
Taking norms we obtain $1=\norm{a}\leq \mu/\lambda+2/\lambda$, and therefore, 
$\lambda-\mu\leq 2$.
Similarly, $\mu-\lambda\leq 2$.
So we already have $\abs{(\lambda-1)-(\mu-1)}\leq 2$.
If furthermore $\lambda+\mu\leq 4$, we also obtain 
$\abs{(\lambda-1)+(\mu-1)}\leq 2$, giving 
$\norm{(\lambda-1)a+(\mu-1)b}\leq \abs{\lambda-1}+\abs{\mu-1}\leq2$.

In the remaining case $\lambda+\mu\geq 4$ we write $(\lambda-1)a+(\mu-1)b$ as 
a non-negative linear combination
\[ (\lambda-1)a+(\mu-1)b = 
    \frac{\lambda+\mu-4}{\lambda+\mu-2}(\lambda a+\mu b) 
  + \frac{2+\lambda-\mu}{\lambda+\mu-2}a
  + \frac{2-\lambda+\mu}{\lambda+\mu-2}b,\]
and apply the triangle inequality:
\[\norm{(\lambda-1)a+(\mu-1)b} \leq 
    2 \frac{\lambda+\mu-4}{\lambda+\mu-2}
    + \frac{2+\lambda-\mu}{\lambda+\mu-2}
    + \frac{2-\lambda+\mu}{\lambda+\mu-2}=2.\]
\end{proof}

\begin{mylemma}\label{euclonlinelemma}
Let $w,a,b$ be vectors in the Euclidean plane such that 
$\norm{w}\leq \sqrt{2}$, $\norm{a}=\norm{b}=1$.
Then there exist signs $\delta,\epsi\in\{\pm 1\}$ such that 
$\norm{w+\delta a+\epsi b}\leq \sqrt{2}$.
\end{mylemma}
\begin{proof}
Note that $a+b\perp a-b$.
Write $p=a+b$, $q=a-b$.
Let $m$ be the midpoint of $pq$, and $L$ the perpendicular bisector of $pq$.
Assume without loss that $\norm{p}\geq\norm{q}$ and that $w$ is inside 
$\angle poq$.
We now show that $\norm{w-p}\leq\sqrt{2}$ or $\norm{w-q}\leq\sqrt{2}$.
Note that as $w$ varies, $\min(\norm{w-p},\norm{w-q})$ is maximized on $L$.
Let $L$ and $op$ intersect in $c$ (between $o$ and $p$), and $L$ and the circle with centre $o$ and radius $\sqrt{2}$ in $d$ (inside $\angle poq$).
See Figure~\ref{fig}.
\begin{figure}
\begin{center}
   \includegraphics{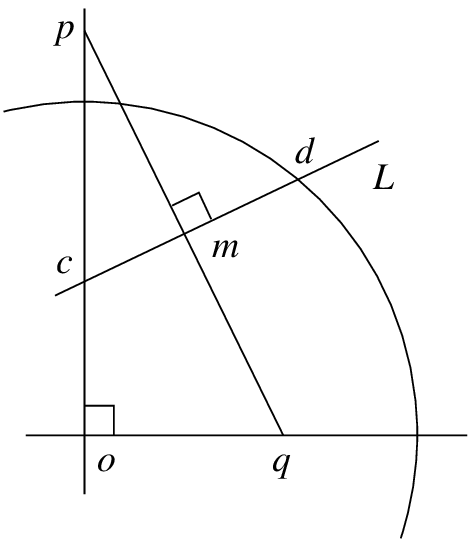}
\end{center}
\caption{}\label{fig}
\end{figure}
Then clearly
\[\max_{\norm{w}\leq\sqrt{2}}\min(\norm{w-p},\norm{w-q}) = \max(\norm{p-c},\norm{p-d}),\]
and we have to show $\norm{p-c}\leq\sqrt{2}$ and $\norm{p-d}\leq\sqrt{2}$.
Since $\norm{p}\geq\norm{q}$, we have $\angle opq\leq 45^\circ$ and $\norm{p-c}=\sec\angle opq\leq\sqrt{2}$.
Since $c$ is between $o$ and $p$, we have $\angle omd\geq 90^\circ$, hence $\norm{m-d}^2\leq \norm{d}^2-\norm{m}^2=2-1$, and 
$\norm{p-d}^2=\norm{p-m}^2+\norm{m-d}^2\leq 1+1$.
\end{proof}

\section{Concluding remarks}
It would be interesting to find higher dimensional generalizations of our results and methods.
We only make the following remarks.

Perhaps there is an analogue of Theorem~\ref{oddunitvectors} with an upper bound of $d-1$ for $n$ unit vectors in a $d$-dimensional normed space where $n\not\equiv d\pmod{2}$.
This would be best possible, as the standard unit vectors in the $d$-dimensional space with the $L_1$ norm show.

Regarding Theorem~\ref{thmtwo}, it is not even clear what the best upper bound in Theorem~\ref{bgtwo} should be.
B\'ar\'any and Grinberg \cite{BG} claim that they can replace $2d$ by $2d-1$.
On the other hand, the upper bound cannot be smaller than $d$, as the $d$-dimensional $L_1$ space shows \cite{BG}.
As the negative part of Theorem~\ref{gametheorem} and the results of \cite{PY} show, an online method would have to have a (sufficiently large) buffer where Player II can put vectors supplied by Player I and take them out in any order.

We finally remark that a naive generalization of Theorem~\ref{ftresult} is not possible, even in Euclidean $3$-space.
For example, using Lemma~\ref{ftchar} it can be shown that for a regular simplex with vertices $x_i \; (i=1,\dots,4)$ there exists a point $x_5$ in the interior of the simplex such that $x_5$ is not a Fermat-Toricelli point of $\{x_1,\dots,x_5\}$ --- we may take any $x_5$ sufficiently near a vertex.
 
\section*{acknowledgments}
We thank the referee for suggestions on improving the paper.

\end{document}